\documentclass[english]{article}
\usepackage[T1]{fontenc}
\usepackage[utf8]{inputenc}
\usepackage{babel}

\usepackage{url}
\usepackage{amsthm}
\usepackage{amsmath}
\usepackage{graphicx}
\usepackage{amssymb}
\usepackage{esint}
\usepackage[unicode=true, pdfusetitle,
 bookmarks=true,bookmarksnumbered=false,bookmarksopen=false,
 breaklinks=false,pdfborder={0 0 1},backref=false,colorlinks=false]
 {hyperref}

\makeatletter
\theoremstyle{plain}
\newtheorem{thm}{Theorem}
  \theoremstyle{definition}
  \newtheorem{defn}[thm]{Definition}
  \theoremstyle{remark}
  \newtheorem{rem}[thm]{Remark}
  \theoremstyle{plain}
  \newtheorem{lem}[thm]{Lemma}

\makeatother

\begin{document}
\global\long\def\P{\mathsf{P}}
 \global\long\def\Pc#1#2{\mathsf{P}\left\{  #1\,\middle|\,#2\right\}  }
 \global\long\def\Pcsq#1#2{\mathsf{P}\left[#1\,\middle|\,#2\right]}
 \global\long\def\E{\mathsf{E}}
 \global\long\def\Ec#1#2{\mathsf{E}\left[#1\,\middle|\,#2\right]}
 \global\long\def\D{\mathsf{D}}
 \global\long\def\Dc#1#2{\mathsf{D}\left[#1\,\middle|\,#2\right]}
 \global\long\def\I{\mathsf{1}}

\title{On short-time asymptotics of one-dimensional Harris flows}

\author{Alexander Shamov}
\maketitle
\begin{abstract}
We study the short-time asymptotical behavior of stochastic flows
on $\mathbb{R}$ in the $\sup$-norm. The results are stated in terms
of a Gaussian process associated with the covariation of the flow.
In case the Gaussian process has a continuous version the two processes
can be coupled in such a way that the difference is uniformly $o\left(\sqrt{t\ln\ln t^{-1}}\right)$.
In case it has no continuous version, an $O\left(\sqrt{t\ln\ln t^{-1}}\right)$
estimate is obtained under mild regularity assumptions. The main tools
are Gaussian measure concentration and a martingale version of the
Slepian comparison principle.


2010 AMS Math subject classification: 60G17, 60G44
\end{abstract}

\section{Introduction}

In this paper we investigate the asymptotical behaviour of the point
motion of one-dimensional stochastic flows. The term {}``stochastic
flow'' means a family of random maps $\left(X_{s,t}\left(\cdot\right)\right)_{s\le t}$
that satisfies the flow property $X_{t,r}\circ X_{s,t}=X_{s,r}$ and
has independent values on disjoint intervals. What we call the point
motion is the family of maps $X_{0,t}$, which we denote by $X\left(\cdot,t\right)$.
We consider only flows of monotone maps from $\mathbb{R}$ to itself.

The basic example of a stochastic flow is a solution of an SDE regarded
as a function of the initial point. Flows of this kind are known to
exist for SDEs with Lipshitz coefficients, and in this case the maps
$X\left(\cdot,t\right)$ are homeomorphisms or even diffeomorphisms
\cite{Kun}. On the other hand, there are also examples of flows of
discontinuous maps \cite{Dar}, the Arratia flow \cite{Arr} being
historically the first of them and perhaps one of the most important.
The point motion of the Arratia flow is a two-parametric process $\left(X\left(u,t\right)\right)_{u\in\mathbb{R},t\ge0}$
such that for each $u$ the process $X\left(u,\cdot\right)$ is a
Brownian martingale with the following properties:
\begin{enumerate}
\item $X\left(u,0\right)=u$
\item $\frac{d}{dt}\left\langle X\left(u,t\right),X\left(v,t\right)\right\rangle =\I\left\{ X\left(u,t\right)=X\left(v,t\right)\right\} $
\item $X\left(u,t\right)\le X\left(v,t\right)$ for all $u\le v$.
\end{enumerate}
\begin{figure}
\includegraphics[width=1\textwidth]{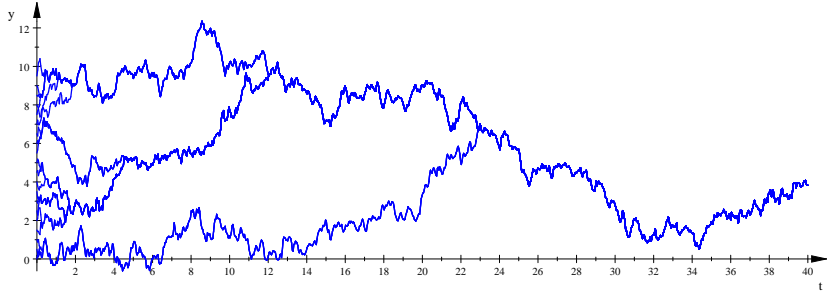}\caption{Point motion of the Arratia flow.\label{fig:ArratiaFlow}}

\end{figure}

Roughly speaking, the Arratia flow consists of Brownian {}``particles''
that evolve independently until they meet, and coalesce thereafter
(Fig. \ref{fig:ArratiaFlow}). It is known that the $X\left(\cdot,t\right)$-image
of any bounded subset of $\mathbb{R}$ is finite for any positive
$t$ due to coalescence \cite{Dor}.

More generally, one can consider so-called Harris flows, defined the
same way except that its {}``infinitesimal covariation function''
may be an arbitrary real positive definite function:\[
\frac{d}{dt}\left\langle X\left(u,t\right),X\left(v,t\right)\right\rangle =\varphi\left(X\left(u,t\right)-X\left(v,t\right)\right).\]
We assume that $\varphi\left(0\right)=1$ for convenience. Furthermore,
we assume that $\left|\varphi\left(x\right)\right|<1$ for $x\neq0$,
which excludes a possibility for periodic flows, regarded more naturally
as flows on the circle. However, taking them into account would lead
to no serious complications.

We study the asymptotical behaviour of \begin{equation}
\sup_{u\in\left[0,1\right]}\left|X\left(u,t\right)-u\right|\label{eq:Sup}\end{equation}
for small $t$. The main approach is to compare $X\left(u,t\right)$
to a family of Gaussian martingales $\left(Y\left(u,t\right)\right)$
which we call a {}``tangent process'', defined by the following
properties:\[
Y\left(u,0\right)=u,\]
\[
\frac{d}{dt}\left\langle X\left(u,t\right),Y\left(v,t\right)\right\rangle =\varphi\left(X\left(u,t\right)-v\right),\]
\[
\frac{d}{dt}\left\langle Y\left(u,t\right),Y\left(v,t\right)\right\rangle =\varphi\left(u-v\right).\]
Note that if $\varphi$ is continuous, then for any fixed $u$ the
quadratic variation of $X\left(u,\cdot\right)-Y\left(u,\cdot\right)$
satisfies\[
\frac{d}{dt}\left\langle X\left(u,t\right)-Y\left(u,t\right)\right\rangle |_{t=0}=0.\]
Since $X\left(u,t\right)-Y\left(u,t\right)$ is a time-changed Brownian
motion \cite{Kal}, one can easily deduce from the law of iterated
logarithm that $\left|X\left(u,t\right)-Y\left(u,t\right)\right|=o\left(\sqrt{t\ln\ln t^{-1}}\right)$
as $t\to0$. It turns out that if $Y$ has a modification that is
continuous w.r.t. both variables then this holds uniformly in $u$.
Namely,\[
\sup_{u\in\left[0,1\right]}\left|X\left(u,t\right)-Y\left(u,t\right)\right|=o\left(\sqrt{t\ln\ln t^{-1}}\right).\]
Together with the law of iterated logarithm for the Gaussian process
$Y$ this yields\[
\limsup_{t\to0}\frac{\sup_{u\in\left[0,1\right]}\left|X\left(u,t\right)-u\right|}{\sqrt{2t\ln\ln t^{-1}}}=1.\]

In Section \ref{sec:MainResult} we consider the case when the {}``tangent
process'' has no continuous modification, which may happen if the
covariation function is not smooth enough at zero. In this case we
compare $X$ and $Y$ in distribution and obtain the following result:\[
\sup_{u\in\left[0,1\right]}\left|X\left(u,t\right)-u\right|-\E\sup_{0\le k<t^{-1/2}}\left|Y\left(kt^{1/2},t\right)-kt^{1/2}\right|=O\left(\sqrt{t\ln\ln t^{-1}}\right).\]
The main tool used there is a martingale version of the Slepian comparison
inequality, well-known in the theory of Gaussian processes \cite{Lif}.
The comparison inequality is stated and proved in Appendix (Theorem
\ref{thm:Comparison}).

The paper is organized as follows. In Section \ref{sec:Existence}
we give basic definitions and state an existence theorem for Harris
flows. In Section \ref{sec:UpperBound} we give a universal $O\left(\sqrt{t\ln t^{-1}}\right)$
upper bound of \eqref{eq:Sup} for monotone families of Brownian motions,
which is used later. In Sections \ref{sec:ContinuousCase} and \ref{sec:MainResult}
we prove our main results for the flows with continuous and discontinuous
tangent processes, respectively. In Appendix we prove the martingale
comparison theorem and a classical result concerning concentration
of measure that is needed in Section \ref{sec:MainResult}.

\section{An existence result \label{sec:Existence}}
\begin{defn}
The point motion of a Harris flow is a family $\left(X\left(u,t\right)\right)_{u\in\mathbb{R},t\ge0}$
of continuous martingales adapted to a common filtration $\left(\mathcal{F}_{t}\right)$,
satisfying the following conditions:
\begin{enumerate}
\item For each $u$ $X\left(u,\cdot\right)$ is an $\mathcal{F}_{t}$-Brownian
motion starting at $u$.
\item For each $u,v$ the joint covariation of $\left(X\left(u,\cdot\right)\right)$
and $\left(X\left(v,\cdot\right)\right)$ is given by\[
\frac{d}{dt}\left\langle X\left(u,t\right),X\left(v,t\right)\right\rangle =\varphi\left(X\left(u,t\right)-X\left(v,t\right)\right),\]
where $\left\langle \cdot,\cdot\right\rangle $ is quadratic covariation,
and $\varphi$ is a positive definite function. \label{enu:Covariation}
\item $\left(X\left(\cdot,t\right)\right)$ is monotone in $u$ for each
$t$, and $\varphi$ is aperiodic. \label{enu:Monotonicity}
\end{enumerate}
\end{defn}
\begin{rem}
Note that condition \ref{enu:Monotonicity} makes the Brownian motions
coalesce once they hit each other.
\end{rem}

\begin{rem}
Once and for all, by $X$ we denote a modification that is separable
and continuous in $t$ for each $u$.
\end{rem}
The following existence resutlt is given in \cite{Har}.
\begin{thm}
The Harris flow exists provided that $\varphi$ is Lipshitz outside
each interval $\left(-c,c\right)$ and its spectral distribution is
not of pure jump type.
\end{thm}
In the sequel we will need not only $X$ itself, but also a Gaussian
process $\left(Y\left(u,t\right)\right)$ starting at $u$ with joint
covariation given below:\[
\frac{d}{dt}\left\langle Y\left(u,t\right),Y\left(v,t\right)\right\rangle =\varphi\left(u-v\right),\]
\begin{equation}
\frac{d}{dt}\left\langle X\left(u,t\right),Y\left(v,t\right)\right\rangle =\varphi\left(X\left(u,t\right)-v\right).\label{eq:DefY}\end{equation}
It admits a construction of the following kind:\[
Y\left(u,t\right)=u+\sum_{i}\intop_{0}^{t}a_{i}\left(u,s\right)dX\left(v_{i},s\right)+\sum_{j}\intop_{0}^{t}b_{j}\left(u,s\right)dW_{j}\left(s\right),\]
where $\left\{ v_{i}\right\} $ is a countable dense subset of $\mathbb{R}$,
$W_{j}$ are independent Brownian motions that are also independent
of $X$, $a$ and $b$ are adapted to the filtration generated by
$X$ and $W$. It is easy to show that $a_{i}$ and $b_{j}$ can be
chosen in such a way that the covariation satisfies \eqref{eq:DefY}.
However, it is not unique, since the construction involves additional
randomization.

\section{An upper bound \label{sec:UpperBound}}

An important special case of a Harris flow is the Arratia flow (Fig.
\ref{fig:ArratiaFlow}). Its covariation function $\varphi$ is given
by $\varphi\left(0\right)=1$ and $\varphi=0$ elsewhere. Thus the
{}``particles'' $X\left(u,\cdot\right)$ move independently until
they coalesce. It follows from our results that the point motion of
the Arratia flow has the following asymptotics in the $\sup$-norm:
\begin{equation}
\sup_{u\in\left[0,1\right]}\left|X\left(u,t\right)-u\right|\sim\sqrt{t\ln t^{-1}},t\to0.\label{eq:TLogT}\end{equation}
Now we will see that the Arratia flow is in some sense the {}``extreme
case''. Namely, for any Harris flow (and in fact for any monotone
family of Brownian motions) an inequality in \eqref{eq:TLogT} holds.
\begin{thm}
For any Harris flow $X$ with $\varphi\left(0\right)=1$ one has\[
\limsup_{t\to0}\sup_{u\in\left[0,1\right]}\frac{\left|X\left(u,t\right)-u\right|}{\sqrt{t\ln t^{-1}}}\le1.\]
\label{thm:UpperBound}\end{thm}
\begin{proof}
First let's prove the inequality for an increasing number of points
$u_{nk}=kt_{n}^{1/2}$, where $t_{n}=q^{n},0<q<1$.\begin{multline*}
\sum_{n}\P\left\{ \sup_{0\le k\le t_{n}^{-1/2}}\sup_{s\le t_{n}}\left|X\left(u_{nk},s\right)-u_{nk}\right|\ge\sqrt{\left(1+\varepsilon\right)t_{n}\ln t_{n}^{-1}}\right\} \le\\
\le\sum_{n}\left\lceil t_{n}^{-1/2}\right\rceil \P\left\{ \sup_{s\le t_{n}}\left|X\left(u_{n0},s\right)-u_{n0}\right|\ge\sqrt{\left(1+\varepsilon\right)t_{n}\ln t_{n}^{-1}}\right\} \le\\
\le\mathrm{const}\cdot\sum_{n}t_{n}^{-1/2}\exp\left[-\frac{1}{2}\left(1+\varepsilon\right)\ln t_{n}^{-1}\right]=\mathrm{const}\cdot\sum_{n}q^{n\varepsilon/2}<+\infty.\end{multline*}
The Borel-Cantelli lemma implies\[
\limsup_{n\to\infty}\sup_{k}\sup_{s\le t_{n}}\frac{\left|X\left(u_{nk},s\right)-u_{nk}\right|}{\sqrt{t_{n}\ln t_{n}^{-1}}}\le1.\]
Now let $u$ be an arbitrary point from $\left[0,1\right]$, and let
$k$ be such that $u_{nk}\le u\le u_{n,k+1}$ for a fixed $t_{n}$.
Using the monotonicity property, we obtain\begin{multline*}
\left|X\left(u,s\right)-u\right|\le\left|X\left(u_{nk},s\right)-u\right|\vee\left|X\left(u_{n,k+1},s\right)-u\right|\le\\
\le\left|X\left(u_{nk},s\right)-u_{nk}\right|\vee\left|X\left(u_{n,k+1},s\right)-u_{n,k+1}\right|+\sqrt{t_{n}}.\end{multline*}
Thus,\[
\limsup_{n\to\infty}\sup_{u\in\left[0,1\right]}\sup_{s\le t_{n}}\frac{\left|X\left(u,s\right)-u\right|}{\sqrt{t_{n}\ln t_{n}^{-1}}}\le1.\]
Now by taking $q$ close enough to $1$ we prove the statement. The
argument is basically the same as in the proof of the law of iterated
logarithm. Namely, let $q$ be such that $\sqrt{q^{n}\ln q^{-n}}\ge\left(1+\varepsilon\right)\sqrt{q^{n+1}\ln q^{-n-1}}$
for sufficiently large $n$. Then since $\sqrt{t\ln t^{-1}}$ is monotone
for small $t$, we obtain\[
\limsup_{t\to0}\sup_{u\in\left[0,1\right]}\frac{\left|X\left(u,t\right)-u\right|}{\sqrt{t\ln t^{-1}}}\le\left(1+\varepsilon\right)\limsup_{t\to0}\sup_{u\in\left[0,1\right]}\frac{\left|X\left(u,t\right)-u\right|}{\sqrt{t_{n}\ln t_{n}^{-1}}}\le1+\varepsilon,\]
where $t_{n}=q^{n}$ is such that $q^{n+1}\le t\le q^{n}$.
\end{proof}

\section{The continuous case \label{sec:ContinuousCase}}

In this paper we estimate the asymptotics of $X$ by comparing it
to the process which we denote $Y$, defined by \eqref{eq:DefY}.
It is a Gaussian process, stationary in $u\in\mathbb{R}$, and also
a Brownian motion in $t$, in the sense that its increments are stationary
and independent. In this section we consider the case when it has
a continuous modification. Note that continuity w.r.t. both variables
follows easily from continuity of $Y\left(\cdot,1\right)$. Indeed,
when restricted to $u\in\left[0,1\right]$ the process becomes a $C\left[0,1\right]$-valued
Brownian motion for which Kolmogorov's continuity criterion is applicable.

A well-known result of the theory of Gaussian processes states that
a stationary Gaussian process has a continuous (or, equivalently,
bounded) modification iff its Dudley integral converges \cite{Lif}.
In our case this is equivalent to\begin{equation}
\intop_{0+}\left|\ln\lambda\left\{ x\,\middle|\,\varphi\left(x\right)\ge1-u^{2}\right\} \right|^{1/2}du<+\infty,\label{eq:ContinityY}\end{equation}
where $\lambda$ is the Lebesgue measure on $\left[0,1\right]$. Note
that continuity of $Y$ does not imply continuity of $X$ %
\footnote{Actually, $X$ is either coalescing or continuous \cite{Mat}, depending
on whether \[
\intop_{0}^{\varepsilon}\frac{x\, dx}{1-\varphi\left(x\right)}\]
is finite. Thus $\varphi\left(x\right)=e^{-\left|x\right|^{\alpha}},0<\alpha<2$
provides an example when $Y$ is continuous but $X$ is not.%
}. Nevertheless, the following result shows that $X$ is close to $Y$
in the $\sup$-norm.
\begin{thm}
Assuming that $Y$ has a continuous modification,\[
\sup_{u\in\left[0,1\right]}\left|X\left(u,t\right)-Y\left(u,t\right)\right|=o\left(\sqrt{t\ln\ln t^{-1}}\right).\]
\label{thm:ContinuousCase}\end{thm}
\begin{proof}
Take a function $\alpha:\left[0,1\right]\to\mathbb{R}_{+}$ that is
monotone, continuous, satisfying $\alpha\left(0\right)=0$ and such
that\begin{equation}
\left\Vert Y\left(\cdot,1\right)\right\Vert _{\alpha}:=\sup_{0\le u<v\le1}\frac{\left|Y\left(u,1\right)-Y\left(v,1\right)\right|}{\alpha\left(\left|u-v\right|\right)}<+\infty\text{ a.s.}\label{eq:AlphaSeminorm}\end{equation}
Its existence may be easily deduced from the fact that the distribution
of $Y\left(\cdot,1\right)$ is supported on a $\sigma$-compact subspace
of $C\left[0,1\right]$. Let $t_{n}$ be $q^{n}$ for some $0<q<1$,
and let's consider $\left\lfloor \ln n\right\rfloor $ points $u_{nk}:=k/\ln n$.
For $Y$ to have a continuous modification, $\varphi$ must be continuous
at zero. Therefore, $X\left(u,\cdot\right)-Y\left(u,\cdot\right)$
are martingales whose quadratic variation is $o\left(t\right)$ uniformly
in $u$:\begin{equation}
V\left(t\right):=\sup_{u\in\left[0,1\right]}\left|\left\langle X\left(u,t\right)-Y\left(u,t\right)\right\rangle \right|=2\sup_{u\in\left[0,1\right]}\left|\intop_{0}^{t}\left(1-\varphi\left(X\left(u,s\right)-u\right)\right)ds\right|=o\left(t\right).\label{eq:VBound}\end{equation}
This implies that $\left|X\left(u,t\right)-Y\left(u,t\right)\right|$
must be $o\left(\sqrt{t\ln\ln t^{-1}}\right)$ for each $u$, and
moreover, uniformly in $u=u_{nk}$, since there are {}``not too many''
of them. More precisely, let $\tau$ be $\inf\left\{ t\,\middle|\, V\left(t\right)>\varepsilon t\right\} $.
One-dimensional continuous martingales are time-changed Brownian motions
\cite{Kal}, hence \begin{multline*}
\sum_{n}\P\left\{ \sup_{0\le k<1/\ln n}\sup_{s\le t_{n}}\left|X\left(u_{nk},s\wedge\tau\right)-Y\left(u_{nk},s\wedge\tau\right)\right|\ge\sqrt{3\varepsilon t_{n}\ln\ln t_{n}^{-1}}\right\} \le\\
\le\mathrm{const}\cdot\sum_{n}\ln n\cdot\exp\left[-\frac{1}{2}\cdot3\ln\ln t_{n}^{-1}\right]\le\mathrm{const}\cdot\sum_{n}\ln n\cdot\left(\ln t_{n}^{-1}\right)^{-3/2}=\\
=\mathrm{const}\cdot\sum_{n}\frac{\ln n}{n^{3/2}}<+\infty.\end{multline*}
By letting $\varepsilon$ be small enough we obtain\[
\sup_{k}\sup_{s\le t_{n}}\left|X\left(u_{nk},s\wedge\tau\right)-Y\left(u_{nk},s\wedge\tau\right)\right|=o\left(\sqrt{t_{n}\ln\ln t_{n}^{-1}}\right).\]
Since $\tau$ is a.s. positive, we may use $X\left(u_{nk},s\right)-Y\left(u_{nk},s\right)$
instead of $X\left(u_{nk},s\wedge\tau\right)-Y\left(u_{nk},s\wedge\tau\right)$.

Points $u\in\left[0,1\right]$ other than $u_{nk}$ may be treated
as follows. Let $k$ be such that $u_{nk}\le u\le u_{n,k+1}$. Then\begin{multline}
\left|X\left(u,s\right)-Y\left(u,s\right)\right|\le\\
\le2\left|Y\left(u_{nk},s\right)-X\left(u_{nk},s\right)\right|+\left|Y\left(u_{n,k+1},s\right)-X\left(u_{n,k+1},s\right)\right|+\\
+\left|Y\left(u_{n,k+1},s\right)-Y\left(u_{nk},s\right)\right|+\left|Y\left(u,s\right)-Y\left(u_{nk},s\right)\right|,s\le t_{n}.\label{eq:GeneralU}\end{multline}
The first two terms in \eqref{eq:GeneralU} are already shown to be
uniformly $o\left(\sqrt{t_{n}\ln\ln t_{n}^{-1}}\right)$. The last
two terms are actually $O\left(\alpha\left(u_{n,k+1}-u_{nk}\right)\sqrt{t_{n}\ln\ln t_{n}^{-1}}\right)$
uniformly in $k$ and $s\le t_{n}$. This follows from the concentration
principle for the $\alpha$-seminorm in \eqref{eq:AlphaSeminorm},
which is in fact valid for any Lipshitz function of a Gaussian random
vector (see Theorem \ref{thm:Concentration} in Appendix). More precisely,
the following inequality holds:\[
\P\left\{ \left\Vert Y\left(\cdot,t\right)\right\Vert _{\alpha}\ge\E\left\Vert Y\left(\cdot,t\right)\right\Vert _{\alpha}+C\right\} \le e^{-C^{2}/2\sigma^{2}t}.\]
for some $\sigma$ and any positive $C$. Together with the fact that
$\E\left\Vert Y\left(\cdot,t\right)\right\Vert _{\alpha}$ is finite
and evidently $O\left(t\right)$, this yields\[
\left\Vert Y\left(\cdot,t_{n}\right)\right\Vert _{\alpha}=O\left(\sqrt{t_{n}\ln\ln t_{n}^{-1}}\right).\]
Therefore,\begin{multline*}
\sup_{s\le t_{n}}\left|X\left(u,s\right)-Y\left(u,s\right)\right|\le o\left(\sqrt{t_{n}\ln\ln t_{n}^{-1}}\right)+\alpha\left(1/\ln t_{n}\right)\left\Vert Y\left(\cdot,t_{n}\right)\right\Vert _{\alpha}=\\
=o\left(\sqrt{t_{n}\ln\ln t_{n}^{-1}}\right).\end{multline*}
$t\neq t_{n}$ are handled in a usual way by letting $q$ close to
$1$.
\end{proof}
Though there are cases when the {}``tangent process'' is discontinuous
and nevertheless the difference $X-Y$ is small enough%
\footnote{We mean not the supremum over $u\in\left[0,1\right]$, which is of
course infinite, but rather the supremum over an increasing number
of points, as considered in Section \ref{sec:MainResult}.%
}, it seems that this is not the case in general. That's why in the
sequel we do not estimate the difference but rather compare the tail
probabilities of $X$ with those of $Y$. In this way we estimate
$\sup_{u\in\left[0,1\right]}\left|X\left(u,t\right)-u\right|$ up
to an $O\left(\sqrt{t\ln\ln t^{-1}}\right)$ term, which is slightly
weaker than the $o\left(\sqrt{t\ln\ln t^{-1}}\right)$ in Theorem
\ref{thm:ContinuousCase}.

\section{Tail comparison \label{sec:MainResult}}

In this section we consider short-time asymptotical behaviour of the
flow with no regularity assumptions on the {}``tangent process''
except local monotonicity of the covariation function. Basically,
we use the same approach as in Theorems \ref{thm:UpperBound} and
\ref{thm:ContinuousCase}. Namely, we start by estimating the deviation
of an increasing number of points $u_{nk}$, and then use the monotonicity
property of the flow to handle the points other than $u_{nk}$. It
turns out that $t_{n}^{-1/2}$ points $u_{nk}=kt_{n}^{1/2}$ give
the right asymptotics up to an $O\left(\sqrt{t_{n}\ln\ln t_{n}^{-1}}\right)$
term.

As it was mentioned earlier, we compare the asymptotical behavior
of the flow to that of a Gaussian process. So first of all, let's
see what happens in the Gaussian case. It is known that the probability
distribution of the supremum of a Gaussian process is concentrated
around its mean at least as strongly as a single Gaussian r.v. is
(see Theorem \ref{thm:Concentration} in Appendix). That is, if $M$
is a centered Gaussian vector in $\mathbb{R}^{d}$, then\begin{equation}
\P\left\{ \left|\sup_{i}M^{i}-\E\sup_{i}M^{i}\right|\ge x\right\} \le Ce^{-x^{2}/2\sigma^{2}}.\label{eq:Concentration}\end{equation}
for some absolute constant $C$ and any $x\ge0$, $\sigma^{2}$ being
$\sup_{i}\E\left(M^{i}\right)^{2}$. From this concentration inequality
it is easy to deduce a law of iterated logarithm of the following
kind:\[
\limsup_{n\to+\infty}\frac{\left|\sup_{k}\left|Y\left(u_{nk},t_{n}\right)-u_{nk}\right|-\E\sup_{k}\left|Y\left(u_{nk},t_{n}\right)-u_{nk}\right|\right|}{\sqrt{2t_{n}\ln\ln t_{n}^{-1}}}\le1.\]
If $Y$ is continuous, then $\E\sup_{k}\left|Y\left(u_{nk},t_{n}\right)-u_{nk}\right|\sim\mathrm{const}\cdot t_{n}^{1/2}$.
In our case, though, the process may be discontinuous, and $\E\sup_{k}\left|Y\left(u_{nk},t_{n}\right)-u_{nk}\right|$
may be asymptotically greater than $\sqrt{t_{n}\ln\ln t_{n}^{-1}}$.
Actually, for the Arratia flow $Y$ consists of independent Brownian
motions%
\footnote{We do not care about separability since in this section we use the
distribution of $Y$ of finite or countable dimension only.%
}, and in this case \[
\sup_{k}\left|Y\left(u_{nk},t_{n}\right)-u_{nk}\right|\sim\E\sup_{k}\left|Y\left(u_{nk},t_{n}\right)-u_{nk}\right|\sim\sqrt{t_{n}\ln t_{n}^{-1}}.\]

We do not know whether a concentration inequality similar to \eqref{eq:Concentration}
holds for $\sup_{u}\left|X\left(u,t\right)-u\right|$. Nevertheless,
we show that $\sup_{u}\left|X\left(u,t\right)-u\right|$ is deterministic
up to $O\left(\sqrt{t\ln\ln t^{-1}}\right)$.
\begin{thm}
Assume that $\varphi$ is monotone on $\left[0,\delta\right]$ for
some $\delta>0$. Then\begin{equation}
\sup_{u\in\left[0,1\right]}\left|X\left(u,t\right)-u\right|=E\left(t\right)+O\left(\sqrt{t\ln\ln t^{-1}}\right),t\to0\text{ a.s.},\label{eq:MainResult}\end{equation}
$E\left(t\right)$ being defined by\[
E\left(t\right)=\E\sup_{0\le k<t^{-1/2}}\left|Y\left(kt^{1/2},t\right)-kt^{1/2}\right|.\]
\end{thm}
\begin{proof}
In the proof we assume that $\varphi$ is monotone on $\left(0,+\infty\right)$.
If $\varphi$ is only locally monotone, the result is obtained for
sufficiently small intervals instead of $\left[0,1\right]$.

First let's prove the upper bound. As usual, take $t_{n}=q^{n}$ and
$u_{nk}=kt_{n}^{1/2}$. For the comparison inequality (Theorem \ref{thm:Comparison})
to be applicable we need a deterministic bound from below on the infinitesimal
covariation of the martingale $\left(X\left(u_{nk},t\right)-u_{nk}\right)$.
If $\varphi$ is monotone on $\left[0,+\infty\right)$, it is sufficient
to obtain a deterministic upper bound on $\sup_{t\le t_{n}}\sup_{k}\left|X\left(u_{nk},t\right)-u_{nk}\right|$.
So we stop the martingale once the deviation gets too large. To be
precise, let's consider the following optional times:\[
\tau_{n}:=\inf\left\{ t\,\middle|\,\sup_{u\in\left[0,1\right]}\left|X\left(u,t\right)-u\right|\ge2\sqrt{t_{n}\ln t_{n}^{-1}}\right\} .\]
Theorem \ref{thm:UpperBound} implies that a.s. $\tau_{n}\ge t_{n}$
for sufficiently large $n$. Take $\tilde{u}_{nk}:=2\left\lceil \sqrt{\ln t_{n}^{-1}}\right\rceil u_{nk}$.
If $\varphi$ is monotone on $\left[0,+\infty\right)$, then the $2\left\lfloor t_{n}^{-1/2}\right\rfloor $-dimensional
martingales $\pm\left(X\left(u_{nk},t\wedge\tau_{n}\right)-u_{nk}\right)$
and $\pm\left(Y\left(\tilde{u}_{nk},t\right)-\tilde{u}_{nk}\right)$
satisfy the conditions of Theorem \ref{thm:Comparison}. Thus\[
\E\exp\lambda\sup_{k}\left|X\left(u_{nk},t_{n}\wedge\tau_{n}\right)-u_{nk}\right|\le\E\exp\lambda\sup_{k}\left|Y\left(\tilde{u}_{nk},t_{n}\right)-\tilde{u}_{nk}\right|\]
for any $\lambda\ge0$ (see also Remark \ref{rem:Submodularity} in
Appendix). Since $\sup_{k}\left|X\left(u_{nk},t\wedge\tau_{n}\right)-u_{nk}\right|$
is a submartingale, the well-known (sub)martingale inequalities \cite{Kal}
imply\begin{equation}
\E\exp\lambda\sup_{t\le t_{n}}\sup_{k}\left|X\left(u_{nk},t\wedge\tau_{n}\right)-u_{nk}\right|\le\mathrm{const}\cdot\E\exp\lambda\sup_{k}\left|Y\left(\tilde{u}_{nk},t_{n}\right)-\tilde{u}_{nk}\right|.\label{eq:ExpX}\end{equation}
The right-hand term may be estimated by means of the concentration
inequality (Theorem \ref{thm:Concentration}):\begin{equation}
\E\exp\lambda\sup_{k}\left|Y\left(\tilde{u}_{nk},t_{n}\right)-\tilde{u}_{nk}\right|\le\exp\left[\lambda\E\sup_{k}\left|Y\left(\tilde{u}_{nk},t_{n}\right)-\tilde{u}_{nk}\right|+t_{n}\lambda^{2}/2\right].\label{eq:ExpYTilde}\end{equation}
What remains is to show that \[
\E\sup_{k}\left|Y\left(\tilde{u}_{nk},t_{n}\right)-\tilde{u}_{nk}\right|=E\left(t_{n}\right)+O\left(\sqrt{t_{n}\ln\ln t_{n}^{-1}}\right),\]
that is, to compare $\E\sup_{k}\left|Y\left(Nu_{nk},t_{n}\right)-Nu_{nk}\right|$
and $\E\sup_{k}\left|Y\left(u_{nk},t_{n}\right)-u_{nk}\right|$, $N$
being equal to $2\left\lceil \sqrt{\ln t_{n}^{-1}}\right\rceil $.
The following inequality is trivial:\[
\sup_{0\le k<t_{n}^{-1/2}}\left|Y\left(Nu_{nk},t_{n}\right)-Nu_{nk}\right|\le\sup_{0\le m<C}S_{m},\]
where\[
S_{m}:=\sup_{0\le k<t_{n}^{-1/2}}\left|Y\left(u_{nk}+mt_{n}^{1/2},t_{n}\right)-u_{nk}-mt_{n}^{1/2}\right|.\]
Note that $S_{m}$ are identically distributed, and also sub-Gaussian
due to the concentration inequality. That is,\[
\E\exp\lambda S_{m}\le\exp\left(\lambda\E S_{m}+t_{n}\lambda^{2}/2\right).\]
What follows is a classical argument that gives an upper bound for
the expectation of supremum of independent sub-Gaussian variables
\cite{Lif}.\begin{multline*}
\E\sup_{m}S_{m}\le\inf_{\lambda}\frac{1}{\lambda}\ln\E\exp\lambda\sup_{m}S_{m}\le\inf_{\lambda}\frac{1}{\lambda}\ln\sum_{m}\E\exp\lambda S_{m}\le\\
\le\inf_{\lambda}\frac{1}{\lambda}\ln\left(N\exp\left(\lambda\E S_{m}+t_{n}\lambda^{2}/2\right)\right)=\inf_{\lambda}\left(\E S_{m}+t_{n}\lambda/2+\frac{\ln N}{\lambda}\right)=\\
=\E S_{m}+\sqrt{2t_{n}\ln N}.\end{multline*}
Since $N\asymp\sqrt{\ln t_{n}^{-1/2}}$, we obtain\begin{equation}
\E\sup_{k}\left|Y\left(\tilde{u}_{nk},t_{n}\right)-\tilde{u}_{nk}\right|\le E\left(t_{n}\right)+O\left(\sqrt{t_{n}\ln\ln t_{n}^{-1}}\right).\label{eq:ExpY}\end{equation}
By combining \eqref{eq:ExpX}, \eqref{eq:ExpYTilde} and \eqref{eq:ExpY},
we obtain\begin{multline*}
\E\exp\lambda\sup_{t\le t_{n}}\sup_{k}\left|X\left(u_{nk},t\wedge\tau_{n}\right)-u_{nk}\right|\le\\
\le\mathrm{const}\cdot\exp\left[\lambda\left(E\left(t_{n}\right)+\mathrm{const}\cdot\sqrt{t_{n}\ln\ln t_{n}^{-1}}\right)+t_{n}\lambda^{2}/2\right].\end{multline*}
Now to estimate the tail probability we may use the Chernoff bound
\cite{Lug}:\begin{multline*}
\P\left\{ \sup_{t\le t_{n}}\sup_{k}\left|X\left(u_{nk},t\wedge\tau_{n}\right)-u_{nk}\right|\ge C+E\left(t_{n}\right)+\mathrm{const}\cdot\sqrt{t_{n}\ln\ln t_{n}^{-1}}\right\} \le\\
\le\mathrm{const}\cdot\inf_{\lambda}e^{-\lambda C+t_{n}\lambda^{2}/2}=\mathrm{const}\cdot e^{-C^{2}/2t_{n}}.\end{multline*}
This implies the upper bound in the law of iterated logarithm for
\[
\sup_{t\le t_{n}}\sup_{k}\left|X\left(u_{nk},t\wedge\tau_{n}\right)-u_{nk}\right|-E\left(t_{n}\right),\]
and since $\tau_{n}\ge t_{n}$ for $n$ sufficiently large, the same
for \[
\sup_{t\le t_{n}}\sup_{k}\left|X\left(u_{nk},t\right)-u_{nk}\right|-E\left(t_{n}\right).\]
The remaining steps are routine.

The lower bound in \eqref{eq:MainResult} is obtained along the same
way. The difference is that now we exchange $u_{nk}$ and $\tilde{u}_{nk}$
to get a bound on the infinitesimal covariation from below.
\end{proof}

\section{Appendix: Comparison and Concentration \label{sec:Appendix}}

The classical comparison inequality due to Slepian says that if $\left(M^{i}\right)$
and $\left(N^{i}\right)$ are centered Gaussian random vectors in
$\mathbb{R}^{d}$ with $\E\left(M^{i}\right)^{2}=\E\left(N^{i}\right)^{2}$
and $\E M^{i}M^{j}\ge\E N^{i}N^{j}$, then $\max_{i}N^{i}$ stochastically
dominates $\max_{i}M^{i}$ \cite{Lif}. For our purpose we need a
generalization involving martingales%
\footnote{Indeed a martingale and a Gaussian martingale.%
} compared by quadratic covariation instead of Gaussian vectors compared
by covariance.

We start with a martingale version of the lemma that is used to derive
comparison inequalities for Gaussian vectors \cite{Lif}.
\begin{lem}
Let $\left(M\left(t\right)\right)_{t\in\left[0,1\right]}$ be a continuous
$\mathbb{R}^{d}$-valued martingale and $\left(N\left(t\right)\right)_{t\in\left[0,1\right]}$
be a continuous $\mathbb{R}^{d}$-valued Gaussian martingale, both
with absolutely continuous quadratic variation and satisfying $M\left(0\right)=N\left(0\right)=0$.
Assume that $N$ is independent of $M$. Then for any $C^{2}$-smooth
function $f:\mathbb{R}^{d}\to\mathbb{R}$ with second derivatives
of at most exponential growth%
\footnote{That is, $\partial_{ij}f\left(x\right)=O\left(\exp\lambda\left\Vert x\right\Vert \right)$
for some $\lambda$. Of course, there must be more natural growth
conditions.%
} the following equality holds:\begin{multline}
\E f\left(M\left(1\right)\right)-\E f\left(N\left(1\right)\right)=\\
=\frac{1}{2}\intop_{0}^{1}\sum_{i,j}\E\partial_{ij}f\left(M\left(t\right)+N\left(1\right)-N\left(t\right)\right)\left(K_{M}^{ij}\left(t\right)-K_{N}^{ij}\left(t\right)\right)dt,\label{eq:Interpolation}\end{multline}
where \[
K_{M}^{ij}\left(t\right)=\frac{d}{dt}\left\langle M^{i}\left(t\right),M^{j}\left(t\right)\right\rangle ,\]
\[
K_{N}^{ij}\left(t\right)=\frac{d}{dt}\left\langle N^{i}\left(t\right),N^{j}\left(t\right)\right\rangle .\]
\label{lem:Interpolation}\end{lem}
\begin{proof}
Let's denote $N\left(1\right)-N\left(1-t\right)$ by $\tilde{N}\left(t\right)$.
Since $N$ is a Gaussian martingale, $\tilde{N}$ is a Gaussian martingale
as well. We may assume that $M$ and $\tilde{N}$ are adapted to independent
filtrations $\left(\mathcal{F}_{t}\right)$ and $\left(\mathcal{G}_{t}\right)$,
respectively. Consider a two-parametric process\[
F\left(t,s\right):=f\left(M\left(t\right)+\tilde{N}\left(s\right)\right).\]
Using Itô's formula w.r.t. $t$ and $s$ separately and taking expectations,
we obtain%
\footnote{Note that since $\left(\mathcal{F}_{t}\right)$ and $\left(\mathcal{G}_{s}\right)$
are independent, by fixing one parameter we obtain (conditionally)
a semimartingale w.r.t. the other one. Thus one-parametric stochastic
calculus is applicable.%
}:\[
\frac{\partial}{\partial t}\E F\left(t,s\right)=\frac{1}{2}\sum_{i,j}\E\partial_{ij}f\left(M\left(t\right)+\tilde{N}\left(s\right)\right)K_{M}^{ij}\left(t\right),\]
\[
\frac{\partial}{\partial s}\E F\left(t,s\right)=\frac{1}{2}\sum_{i,j}\E\partial_{ij}f\left(M\left(t\right)+\tilde{N}\left(s\right)\right)K_{N}^{ij}\left(1-s\right).\]
Therefore,\[
\frac{\partial}{\partial t}\E F\left(t,1-t\right)=\frac{1}{2}\sum_{i,j}\E\partial_{ij}f\left(M\left(t\right)+\tilde{N}\left(1-t\right)\right)\left(K_{M}^{ij}\left(t\right)-K_{N}^{ij}\left(t\right)\right).\]
Finally, by integrating over $t\in\left[0,1\right]$ we finish the
proof.
\end{proof}
Now suppose that we are given an inequality between $K_{M}^{ij}$
and $K_{N}^{ij}$. It is then clear that by means of Lemma \ref{lem:Interpolation}
we may obtain an inequality between $\E f\left(M\left(1\right)\right)$
and $\E f\left(N\left(1\right)\right)$ for an appropriate class of
functions.
\begin{thm}
[Martingale comparison] Let $M$ and $N$ be a martingale and a Gaussian
martingale with absolutely continuous quadratic variation, and let
$f:\mathbb{R}^{d}\to\mathbb{R}$ be a Borel function of at most exponential
growth. Assume that the following inequalities hold%
\footnote{Derivatives of $f$ are understood in the sense of Schwartz distributions.%
}: \[
K_{M}^{ii}+K_{M}^{jj}-2K_{M}^{ij}\le K_{N}^{ii}+K_{N}^{jj}-2K_{N}^{ij},i\neq j,\]
\[
K_{M}^{ii}\le K_{N}^{ii},\]
\begin{equation}
\partial_{ij}f\le0,i\neq j.\label{eq:Submodularity}\end{equation}
Furthermore, assume that either one of the following additional conditions
is fulfilled:
\begin{enumerate}
\item \[
K_{M}^{ii}=K_{N}^{ii}\]

\item \begin{equation}
\sum_{j}\partial_{ij}f\ge0\text{ for each }i\label{eq:Convexity}\end{equation}

\end{enumerate}
Then\[
\E f\left(M\left(1\right)\right)\le\E f\left(N\left(1\right)\right).\]
\label{thm:Comparison}\end{thm}
\begin{proof}
Assume that the second derivatives of $f$ are continuous and of at
most exponential growth. Then by Lemma \ref{lem:Interpolation} we
have\[
\E f\left(M\left(1\right)\right)-\E f\left(N\left(1\right)\right)=\frac{1}{2}\intop_{0}^{1}\sum_{i,j}\E\partial_{ij}f\left(M\left(t\right)+N\left(1\right)-N\left(t\right)\right)\left(K_{M}^{ij}-K_{N}^{ij}\right)dt.\]
Note that in order to use Lemma \ref{lem:Interpolation} we assume
that $M$ and $N$ are independent. If they are not, we may replace
$N$ by an independent process with the same distribution.

Next we rewrite the right-hand side in the following way:\begin{multline*}
\sum_{i,j}\partial_{ij}f\cdot\left(K_{M}^{ij}-K_{N}^{ij}\right)=\\
=\sum_{i<j}\partial_{ij}f\cdot\left[\left(2K_{M}^{ij}-K_{M}^{ii}-K_{M}^{jj}\right)-\left(2K_{N}^{ij}-K_{N}^{ii}-K_{N}^{jj}\right)\right]+\\
+\sum_{i}\left(\sum_{j}\partial_{ij}f\right)\left(K_{M}^{ii}-K_{N}^{ii}\right).\end{multline*}
The conditions imposed upon $f$ and $K_{M}-K_{N}$ ensure that each
term is negative.

The case when $f$ is not smooth enough may be treated by means of
an approximation argument. Namely, let $\varphi_{\varepsilon}\in C^{\infty}\left(\mathbb{R}^{d}\to\mathbb{R}\right)$
be a nonnegative function supported on $\left\{ \left\Vert x\right\Vert \le\varepsilon\right\} $,
such that $\intop\varphi_{\varepsilon}dx=1$. Then $f\ast\varphi_{\varepsilon}$
satisfies the conditions of Lemma \ref{lem:Interpolation}, and $f\ast\varphi_{\varepsilon}$
converges to $f$ in $L^{1}$ over any Gaussian measure due to the
growth condition.\end{proof}
\begin{rem}
The basic condition \eqref{eq:Submodularity} is referred to as submodularity
or $L$-subadditivity. It is known to be equivalent to the following
inequality that involves only the lattice structure:\[
f\left(x\wedge y\right)+f\left(x\vee y\right)\le f\left(x\right)+f\left(y\right)\text{ for all }x,y\in\mathbb{R}^{d}.\]
Here $x\wedge y$ and $x\vee y$ are coordinatewise minimum and maximum,
respectively. Examples of submodular functions include $f\left(x^{1},\dots,x^{d}\right)=\varphi\left(\max_{i}x^{i}\right)$
for any increasing function $\varphi$. If $\varphi$ is also convex,
then $f$ satisfies \eqref{eq:Convexity}.\label{rem:Submodularity}
\end{rem}

\begin{rem}
It is clear that $M$ and $N$ may be exchanged, as long as integrability
issues are taken care of.%
\footnote{In the case of our interest nothing bad happens, since the martingale
is bounded.%
} Thus we also have comparison inequalities in the case when the infinitesimal
covariation of a martingale is bounded deterministically from below.
\label{rem:LowerBound}
\end{rem}
Next we present the basic result concerning concentration of measure
for Lipshitz functionals of Gaussian random vectors. What follows
is a short proof based on martingale comparison%
\footnote{Though, the comparison principle is used in the one-dimensional setting,
which is rather trivial.%
} \cite{Led}. Another approach based on the isoperimetric properties
of Gaussian measures may be found in \cite{Led,Lif}.
\begin{thm}
[The concentration principle] Let $N$ be a standard Gaussian random
vector in $\mathbb{R}^{d}$, and let $f$ be a Lipshitz function with
Lipshitz constant $L$. Then the following inequalities hold:\begin{equation}
\E\exp\lambda\left(f\left(N\right)-\E f\left(N\right)\right)\le\exp\left(\lambda^{2}L^{2}/2\right),\forall\lambda\in\mathbb{R},\label{eq:ConcentrationExp}\end{equation}
\begin{equation}
\P\left\{ f\left(N\right)-\E f\left(N\right)\ge C\right\} \le\exp\left(-C^{2}/2L^{2}\right),\forall C\ge0.\label{eq:ConcentrationTail}\end{equation}
\label{thm:Concentration}\end{thm}
\begin{proof}
Let $\left(N\left(t\right),0\le t\le1\right)$ be a standard Brownian
motion in $\mathbb{R}^{d}$ with $N=N\left(1\right)$. Denote by $\mathcal{F}_{t}$
the induced filtration. We consider the martingale \[
\Phi\left(t\right):=\Ec{f\left(N\right)}{\mathcal{F}_{t}}\]
and intend to prove that\begin{equation}
d\left\langle \Phi,\Phi\right\rangle \le L^{2}dt.\label{eq:dPhi}\end{equation}
By an application of Theorem \ref{thm:Comparison} to $\Phi-\E f\left(N\right)$
and the Brownian motion in $\mathbb{R}$ with quadratic variation
$L^{2}t$, this would imply \eqref{eq:ConcentrationExp}. To bound
the tail probability in \eqref{eq:ConcentrationTail} we may then
use the classical Chernoff bound \cite{Lug}:\begin{multline*}
\P\left\{ f\left(N\right)-\E f\left(N\right)\ge C\right\} \le\inf_{\lambda\ge0}e^{-\lambda C}\E\exp\lambda\left(f\left(N\right)-\E f\left(N\right)\right)\le\\
\le\inf_{\lambda\ge0}\exp\left(-\lambda C+\lambda^{2}L^{2}/2\right)=\exp\left(-C^{2}/2L^{2}\right).\end{multline*}
What remains is to prove \eqref{eq:dPhi}. For this we note that\[
\Ec{f\left(N\left(1\right)\right)}{\mathcal{F}_{t}}=\Ec{f\left(N\left(1\right)\right)}{N\left(t\right)}=T^{1-t}f\left(N\left(t\right)\right),\]
where $T$ is the Brownian semigroup. The stochastic differential
$dT^{1-t}f\left(N\left(t\right)\right)$ can be calculated using Itô's
formula. Note that the $dt$ terms vanish automatically since $\Phi$
is a martingale, and just the $dN$ term remains:\[
dT^{1-t}f\left(N\left(t\right)\right)=\sum_{i}T^{1-t}\partial_{i}f\left(N\left(t\right)\right)dN^{i}\left(t\right).\]
Now the Lipshitz condition implies \eqref{eq:dPhi}.\end{proof}
\begin{rem}
Of course, Theorem \ref{thm:Concentration} may be formulated for
any Gaussian random vector, not just a standard one. In this case
the Lipshitz condition is assumed w.r.t. the Euclidean metric induced
by the Gaussian measure.\end{rem}

\end{document}